\documentclass[12pt,a4paper]{amsart}
\usepackage{times}
\usepackage{amscd,amssymb,amsthm,latexsym,amsmath,stackrel}
\usepackage{amsfonts,latexsym,amssymb,url}
\usepackage{pifont}
\usepackage{color}

\newtheorem{theorem}{Theorem}[section]
\newtheorem{cor}{Corollary}[section]
\newtheorem{lemma}{Lemma}[section]

\newtheorem{defi}{Definition}[section]
\newtheorem{remark}{Remark}[section]

\newcommand{\R}{{\mathbb R}}
\newcommand{\Z}{{\mathbb Z}}
\newcommand{\N}{{\mathbb N}}

\begin{document}
   \title[On Burenkov's extension operator]
   { On Burenkov's extension operator preserving  Sobolev-Morrey spaces on Lipschitz domains}
   \author{Maria Stella Fanciullo}
   \address{
Dipartimento di Matematica e Informatica, Universit\`a degli Studi di Catania, Viale Andrea Doria 6, 95125, Catania}
   \email{fanciullo@dmi.unict.it}
\author{Pier Domenico Lamberti}
   \address{Dipartimento di Matematica, Universit\`a degli Studi di Padova, via Trieste 63, 
 35121,  Padova}
   \email{lamberti@math.unipd.it}

   \thanks{}
\date{\today}

\linespread{1.3}

\begin{abstract} We prove that Burenkov's Extension Operator preserves So\-bo\-lev spaces built on general Morrey spaces, including classical Morrey spaces.
The analysis concerns bounded and unbounded open sets with Lipschitz boundaries in the n-dimensional Euclidean space.  
\end{abstract}
\maketitle
%\tableofcontents

\section{Introduction}

The extension problem is a classical problem in the theory of function spaces with important applications in many fields of mathematical analysis, in particular  harmonic analysis 
and the theory of partial differential equations.  Broadly speaking, the problem consists in extending to the whole of $\R^n$ the elements of a space of functions defined on a given 
subset of $\R^n$, with  preservation of  certain differentiability and summability  properties. The analysis of such problem goes back to the works of  Whitney~\cite{whytannals, whyt} and Hestenes~\cite{hest} 
who considered   spaces of continuously differentiable functions.  In the case of Sobolev spaces $W^{l,p}(\Omega )$, with $l\in \N$ and $p\in [1,\infty ]$, defined on  open sets $\Omega$ in $\R^n$ with minimal boundary regularity, i.e. $\Omega $ in the Lipschitz  class $C^{0,1}$, the problem received 
important contributions by Calderon~\cite{cal}, Stein~\cite{steinorsay, steinbook} and Burenkov~\cite{burpaper2, burpaper1}. They constructed three different linear bounded extension operators from $W^{l,p}(\Omega)$ to $W^{l,p}(\R^n)$. 
Compared with the classical extention operator by Hestenes~\cite{hest}, the main striking feature of Calderon's, Stein's and Burenkov's operators consists in the fact that $\Omega $ is not required to be of class $C^l$ with $l>1$. For a discussion concerning the differences between those operators, as well as for historical remarks and other references, we refer to Burenkov~\cite{burpaper1, b}, and to  the earlier Stein's book~\cite{steinbook}. For the convenience of the reader,  we briefly mention here the main properties of such operators.   
Calderon's Extension Operator is based on an integral representation formula involving singular integral operators, hence it  does not allow to deal with the cases $p=1,\infty $. Stein's Extension Operator concerns all exponents $p\in [1,\infty ]$ and is universal in the sense that the same operator can be used for all orders of smoothness $l$. Burenkov's Extension Operator is not universal but allows dealing with all exponents   $p\in [1,\infty ]$, and also with anisotropic Sobolev spaces. Moreover,  Burenkov's  operator provides functions which are $C^{\infty }$ outside $\Omega $ and the order of growth of the derivatives of such functions, when approaching the boundary, is the best possible in some sense. We also mention 
that Burenkov's  operator allows to deal  with  open sets of class $C^{0,\gamma}$ with $\gamma <1$, in which case the target space is not $W^{l,p}(\R^n)$ but $W^{\gamma l, p}(\R^n)$. 

We denote  Burenkov's Extension Operator by $T$ and we refer to formula (\ref{burext}) for its definition in the case of an elementary Lipschitz domain $\Omega$ given by the subgraph of a Lipschitz function, and to formula   (\ref{burextgen}) for the case of general bounded or unbounded Lipschitz domains. For simplicity, we do not emphasize the dependence of $T$ on $l$ in the notation, but it is always understood.    Burenkov's Extension Operator is also described in great detail in Burenkov's book~\cite[Chap.~6]{b}, to which we shall refer in this paper for any result required in our proofs. 

We note that  the operator $T$ in (\ref{burext})   is defined by means of a sequence of mollifiers with variable steps and has a local nature in the sense that the values of the extended function $Tf$  around  a point in $ \R^n\setminus \Omega $ depend only on  the values of $f$  localized around certain `reflected' points  inside $\Omega$.   

The main aim of the present paper is to exploit the local nature of Burenkov's Extension Operator in order to prove that such operator preserves also  Sobolev-Morrey spaces. 

Given $p\in [1,\infty [$,  a  function $\phi : ]0,\infty [\to ]0,\infty [$ and $\delta \in ]0, \infty ]$, for all $f\in L^p_{loc}(\Omega )$, we set
\begin{equation*}\label{norm_generalized_morrey}
\| f \|_{M^{\phi , \delta }_p(\Omega ) }:=\sup_{x\in \Omega}\sup_{0<r<\delta }\left(\frac{1}{\phi(r)}\int_{B_r(x)\cap \Omega}|f|^pdy \right)^{\frac{1}{p}}.
\end{equation*}
We also write $M^{\phi  }_p(\Omega ) $ to denote $M^{\phi , \infty }_p(\Omega ) $. The Morrey space $M^{\phi , \delta }_p(\Omega ) $ is the space of functions 
$f\in L^p_{loc}(\Omega )$ such that $\| f \|_{M^{\phi , \delta }_p(\Omega ) }<\infty $. If $\phi (r) =r^{\lambda}  $ for all $r\in ]0,\infty [$, we obtain the classical Morrey spaces $M^{\lambda , \delta}_p(\Omega )$ (which are known to be of interest only in the case $\lambda \in [0, n]$). 

Given $l\in \N$, $p\in [1,\infty[$ and $\delta \in [0,\infty [$,  the main result of the paper is the following estimate
\begin{equation}\label{finalthm1intro}
\| D^{\alpha}Tf\|_{M^{\phi ,\delta }_p(\R^n)}\le c\sum_{0\le |\beta |\le |\alpha|} \| D^{\beta}f\|_{M^{\phi , \delta}_p(\Omega)},
\end{equation}
for all  $f\in W^{l,p}(\Omega)$ and $|\alpha|\le l$, where $c>0$ is independent of $f$. See Theorem~\ref{finalthm}.  Moreover, we  also prove that if $\Omega $ is a bounded or an elementary unbounded domain,  then $c$ can be chosen to be independent of $\delta$, hence in these cases  estimate  $(\ref{finalthm1intro})$ holds also if $\delta =\infty $. See Corollary~\ref{morrey} for the case of elementary unbounded domains.

In particular, if $f\in W^{l,p}(\Omega )$ is such that $D^{\alpha }f\in M^{\phi , \delta }_p(\Omega )$ for all $|\alpha |\le l$ then $D^{\alpha }Tf\in  M^{\phi , \delta }_p(\Omega )$ for all $|\alpha |\le l$. 

This paper is organized as follows. In section 2, we consider the case of elementary Lipschitz domains defined by the subgraphs of Lipschitz continuous functions. Section 3 is devoted to
the case of general Lipschtz open sets. 

For another contribution in this field of investigation, we refer to Khidr and Yeihia~\cite{khi} which obtain results radically different from ours.  

%{\color{red} mi domando se se sia il caso o meno di riferire anche ad Acquistapace??? forse anche no}

\section{Burenkov's Extension Operator on elementary Lipschitz domains}

In this paper the elements of $\R^n$ are denoted by  $x=(\overline x,x_n)$ with $\overline x\in\R^{n-1}$ and $x_n\in \R$. 

 If $\Omega$ is an open subset of $\R^n$, we denote by $W^{l,p}(\Omega)$  the Sobolev space of function $f\in L^p(\Omega)$ with weak derivatives $D^{\alpha}f\in L^p(\Omega)$  for all  $|\alpha |\le l $, endowed with the norm $\| f\|_{W^{l,p}(\Omega)}=\sum_{0\le |\alpha |\le l}\|D^{\alpha}f\|_{L^p(\Omega)}$.    

\subsection{The case of unbounded Lipschitz subgraphs}

In this subsection we consider elementary Lipschitz  domains   $\Omega$  in $\R^n$ of the form

\begin{equation}\label{eldom}\Omega = \{x=(\overline x, x_n)\in \R^n: x_n<\varphi(\overline x),\overline x\in \R^{n-1}\}\,,
\end{equation}
where $\varphi:\R^{n-1}\to \R$ is a Lipschitz continuous function, that is
there exists a positive constant $M$ such that
\begin{equation}\label{lip1}
|\varphi(\overline x)-\varphi(\overline y)| \le M|\overline x-\overline y|\,, \,\,\forall \overline x, \overline y\in \R^{n-1}\,.
\end{equation}

The best constant $M$ in inequality (\ref{lip1}) is the Lipschitz constant of $\varphi$ and is denoted by ${\rm Lip} \varphi$.

Let $G=\R^n\setminus \overline \Omega$.  For every  $k\in\Z$, we set
$$G_k=\{x\in G: 2^{-k-1}<\rho_n(x)\le 2^{-k}\}$$
where $\rho_n(x)=x_n-\varphi(\overline x)$
is the signed distance from $x\in \R^n$ to $\partial  G$ in the $x_n$ direction. Clearly,  $\rho_n(x)\ge 0$ for all $ x\in G$.

In the sequel, we need the following  Partition of Unity's Lemma from  \cite[Lemma~18]{b}.   Here  $\N_0$ denotes the set of natural numbers including zero. 

\begin{lemma}\label{partition}
There exists a sequence of nonnegative functions $\psi_k$ belonging to $C^{\infty}(\R^n)$, for all $k\in \Z$,  satisfying the following conditions:
\begin{itemize}
\item[(i)]
$\displaystyle{ \sum_{k=-\infty}^\infty\psi_k=\begin{cases}1, \,\, {\rm if}\,\, x\in G,\vspace{2mm}\\ 
0, \,\, {\rm if }\,\, x\notin G;\end{cases}}$

\item[(ii)]  $  G=\displaystyle \cup_{k=-\infty}^\infty {\rm supp} \psi_k $
and  the covering $\{{\rm supp} \psi_k\}_{k\in \Z}$ has multiplicity equal to $2$;

\item[(iii)]  $G_k\subset {\rm supp} \psi_k\subset G_{k-1}\cup G_k\cup G_{k+1}$, for all $k\in \Z$;

\item[(iv)]  $|D^\alpha\psi_k(x)|\le c(\alpha) 2^{k|\alpha|}$, for all $x\in \R^n,k\in \Z, \alpha\in \N^n_0$. 
\end{itemize}
\end{lemma}

 We are now ready to  recall the definition of  Burenkov's Extension Operator for an elementary Lipschitz  domain $\Omega$ as in (\ref{eldom}). 

Let $l\in\N$ and $1\le p\le \infty$. For every $f\in W^{l,p}(\Omega)$, we set 
\begin{equation}
\label{burext}(Tf)(x)=\begin{cases}f(x), \,\, {\rm if}\,\,x \in \Omega,\vspace{2mm}\\
\displaystyle\sum_{k=-\infty}^\infty \psi_k(x)f_k(x),  \,\, {\rm if}\,\,x \in G,
\end{cases}
\end{equation}
where 
\begin{multline*}
f_k(x)=\int_{\R^n}f(\overline x -2^{-k}\overline z,x_n-A2^{-k}z_n)\omega(z)dz =\\
=A^{-1}2^{kn}\int_{\R^n}\omega(2^k(\overline x-\overline y), A^{-1}2^{k}(x_n-y_n))f(y)dy\, ,
\end{multline*}  
$A=16(M+1)$, and $\omega\in C^\infty_c(\R^n)$ is a kernel mollification satisfying 
$${\rm supp}\, \omega\subset\left\{ x\in \overline {B(0,1)}: x_n\ge \frac{1}{2}\right\}\,, $$
$\int_{B(0,1)}\omega(z)dz=1\,$ and  $\int_{B(0,1)}\omega(z)z^\alpha dx=0$, $\alpha\in \N^n_0$, $0<|\alpha|\le l$.

The operator $T$ is a linear continuous operator from $W^{l,p}(\Omega)$ to $W^{l,p}(\R^n)$,  see \cite[p.~286]{b}.  

Following \cite{b}, for every $k\in \Z$ we set 
$$\tilde G_k=G_{k-1}\cup  G_k\cup  G_{k+1} =\{x\in G: 2^{-k-2}<\rho_n(x)\le 2^{-k+1}\}\, .$$

We prove now the following 

%{\color{red} Importante: per poter applicare questo lemma a quello successivo, abbiamo bisogno della condizione $B_r \cap %{\color{blue} \tilde G_k  }\ne\emptyset$ e non della piu' forte  $B_r \cap  G_k  \ne\emptyset$  come era scritto prima.  %Fortunatamente, anche con la condizione blu piu' debole, questo lemma funziona} 

\begin{lemma}\label{controllo_geo}
Let $B_r$ be a ball in $\R^n$ of radius $r$ such that $B_r \cap G\ne\emptyset$. Let $h\in \Z$ be the minimum integer such that $B_r \cap G_h\ne\emptyset$. Let $k\in \Z$  be such that $k\ge h+3$ and $B_r \cap  \tilde G_k  \ne\emptyset$. Then 
$$
|2^{-(h+3)}-2^{-k}|\le r(M+1)\,,
$$
where $M$ is in \eqref{lip1}.
\end{lemma} 
\begin{proof}
Since  $B_r \cap G_h, B_r \cap  \tilde G_k \ne\emptyset$ and $k\geq h+3$, it follows that $\{x\in B_r:\, \rho_n(x)=2^{-h-2} \}  , \{x\in B_r:\, \rho_n(x)=2^{-k+1} \}\ne \emptyset$. Hence, by taking $y,w\in B_r$ with $y_n-\varphi (\bar y)=2^{-h-2}$ and $w_n-\varphi(\bar w)=2^{-k+1} $
 we have
\begin{multline*}
|2^{-(h+3)}-2^{-k}|=\frac{1}{2}|2^{-h-2}-2^{-k+1}|=\frac{1}{2}|y_n-\varphi (\bar y)-w_n+\varphi (\bar w)|\\
\le \frac{1}{2}(|y_n-w_n|+M|\bar y-\bar w|)\le r(M+1).
\end{multline*}
\end{proof}

Then we need the following lemma. Here and in the sequel by $B_r(x)$ we denote a ball with center $x$ and radius $r$. Moreover, $M$ is the constant in $(\ref{lip1})$.

\begin{lemma}\label{inc}
Let $B_r$ and $h\in \Z$ be as in Lemma \ref{controllo_geo}, and $E>0$.  Then  there exists a positive constant $S$ depending only on $M, E$  such that      for every $\eta\in \R^n$, with $|\eta|<E$,   there exists  a ball $B_{Sr}(x_\eta)$, such that
\begin{equation}\label{inc0}
\bigcup_{k=h+3}^\infty \left( B_r\cap \tilde G_k- 2^{-k}\eta\right)\subset B_{Sr}(x_\eta)\,.
\end{equation}

Moreover, there exist $K\in \N$ depending only on $n,M,E$, and $K$ balls $B_r(x_\eta^{(i)})$, $i=1,\dots, K$, such that 
\begin{equation}\label{inc1}
\bigcup_{k=h+3}^\infty \left( B_r\cap \tilde G_k- 2^{-k}\eta\right)\subset \bigcup_{i=1}^{K}B_{r}(x^{(i)}_\eta).
\end{equation} 

Finally, $x_\eta$ and $x_{\eta}^{(i)}$ can be chosen to depend with continuity on $\eta$ for all $i=1,\dots , K$. 
\end{lemma}

\begin{proof}
We suppose directly that  $B_r \cap \tilde G_{h+3}\ne\emptyset$, otherwise the unions in the left hand-sides of \eqref{inc0} and \eqref{inc1} are empty, and the statement is trivial.
Let $k\ge h+3$ be such that  $B_r \cap \tilde G_k\ne\emptyset$. Let $a\in B_r \cap \tilde G_{h+3}$ and $b\in B_r \cap \tilde G_{k}$. Then  by Lemma \ref{controllo_geo}, for all $ \eta\in \R^n$, with $|\eta| <E$,  we have 
\begin{multline*}
|b- 2^{-k}\eta- (a-2^{-(h+3)}\eta)|\le |b-a|+ |2^{-k}-2^{-(h+3)}||\eta|\\
\le [2+(M+1)E]r\,.
\end{multline*}
 
 Then, choosing $S=2[2+(M+1)E]$, the ball $B_{Sr}(x_{\eta})$, with radius $Sr$ and center $x_\eta=a-2^{-(h+3)}\eta$, satisfies inclusion \eqref{inc0}.

 Finally, it is obvious that each ball $B_{Sr}(x_\eta)$ can be covered by a finite number of balls of radius $r$ as in \eqref{inc1}. Moreover, the fact that the centers $x_\eta^{(i)}$ can be chosen to depend with continuity on $\eta$ can be deduced by the continuous dependence on $\eta$ of $x_{\eta}=a-2^{-(h+3)}\eta$ via a simple but lenghty argument which is not worth including here.
\end{proof}

As in \cite[Chap.~6]{b}, for every $k\in \Z$ we set 
\begin{equation*}\tilde \Omega_k=\{x\in \Omega: 2^{-k-2}<|\rho_n(x) |\le b2^{-k+1}\}\,, 
\end{equation*}
where $b=10A$.

We can now prove the following lemma.

%{\color{red} A parte la questione del numero di palle che dipendeva da d, mi sono accorto che in realta' il d da usare nel lemma %successivo non era la distanza di $B_r$ da $\Omega$ bensi' il sup delle distanze tra i punti di  detta palla e $\Omega$. Questo %avrebbe dei riflessi negativi nel seguito perche' darebbe vita a una dipendenza dal raggio r. Per evitare questa conseguenza, ho %introdotto un insieme ausiliario ${\mathcal{U}}$ che serve a tenere a bada la eventuale esplosione di r. Nel seguito il ruolo di %${\mathcal{U}}$ sara' giocato dal "riflesso" del supporto della funzione f   }

In the proof of the following lemma and of the other statements in the sequel, the value of the constant $c$ may change from line to line but is always independent of the function $f$ and of the radius $r$.

\begin{lemma}\label{mainlem}
Let $f\in W^{l,p}(\Omega)$ and $B_r$ a ball in $\R^n$  of radius $r$ such that $B_r \cap G\ne \emptyset$. The following statements hold:
\begin{itemize}
\item[(i)]  There exists $c>0$ depending only on $n$, $l$, $p$, $M$, $\omega$ and there exists $H\in \N$, depending only on $n$ and $M$ such that for every $z\in B_1(0)$, with  $z_n>1/2$, there exist $H$ balls  $B_{ r}(x^{(i)}_z)$,  $i=1,\dots, H$,  such that  
\begin{equation}\label{mainlem1}
\|D^\alpha f_k\|_{L^p(B_r\cap \tilde G_k)}\le c\int_{   \{z\in B_1(0):\ z_n>1/2\}   } \|D^\alpha f\|_{L^p(\cup_{i=1}^H   B_{ r}(x^{(i)}_z) \cap \tilde \Omega_k) }dz\,,
\end{equation}
for  all $k\in \Z$, and  $\alpha \in \N^n_0$, with $|\alpha|\le l$. 
\item[(ii)]
Let ${\mathcal{U}}\subset \R^n$ be a  fixed measurable set and let $d=\sup\{\rho_n(x):\ x\in B_r\cap {\mathcal{U}} \}$.  
 Assume that $d<\infty $.
There exists  $c>0$ depending only on $n$, $l$, $p$, $M$,  $\omega$,  there exists $H_{{\mathcal{U}}}\in \N$ depending only on $n, M, d$, and  
for every $\alpha \in \N^n_0$  with $|\alpha |\le l$ there exists a function $g_\alpha$ independent of $r, {\mathcal {U}}$, such that 
for every $z\in B_{1+cd}(0)$ there exist $H_{{\mathcal{U}}}$   balls  $B_{ r}(x^{(i)}_{z})$,  $i=1,\dots,H _{{\mathcal{U}}} $,  such that  
\begin{multline*}%\label{mainlem2}
\| D^{\alpha}f_k-g_{\alpha}\|^p_{L^p(B_r\cap {\mathcal{U}}\cap \tilde G_k)}\\
\le c2^{pk(|\alpha |-l)} \int_{B_{1+cd}(0) } \sum_{|\beta |=l}  \| D^{\beta }f\|^p_{L^{p}(\cup_{i=1}^{H_{\mathcal{U}}}   B_{ r}(x^{(i)}_z)\cap \tilde \Omega_k) } dz,  
\end{multline*}
for all $k\in \Z$.  
\end{itemize}
Moreover, in both statements  points $x^{(i)}_z$ can be chosen to depend with continuity on $z$. 
\end{lemma}
\begin{proof}
We begin with proving statement (i). By differentiating under integral sign  and using Minkowskii inequality we get
\begin{multline}\label{mainlem3}
\|D^\alpha f_k\|_{L^p(B_r\cap \tilde G_k)}\\
\le c \int_{\{z\in B_1(0):\ z_n>1/2\}} \|D^\alpha f(\overline x-2^{-k}\overline z, x_n-A2^{-k}z_n)\|_{L^p(B_r\cap \tilde G_k)}dz\\
=c\int_{\{z\in B_1(0):\ z_n>1/2\}    } \|D^\alpha f\|_{L^p(B_r\cap \tilde G_k-2^{-k}(\overline z, Az_n))}dz.
\end{multline}
 Let $h\in \Z$ be the minimum integer such that $B_r \cap G_h\ne\emptyset$. By Lemma~\ref{inc}, there exists $K\in \N$ depending only on $n$ and $M$ such that for every $z\in \R^n$, $|z|<1$ there exist $K$  balls $B_{r}(x^{(i)}_{z})$ such that 
\begin{equation*}\label{mainlem4}
\bigcup_{k\in \Z}(B_r\cap \tilde G_k-2^{-k}(\overline z, Az_n))\subset 
\bigcup_{ k=h-1}^{h+2}(B_r\cap \tilde G_k-2^{-k}(\overline z, Az_n))\bigcup_{i=1}^K B_{r}(x^{(i)}_{z}),
\end{equation*}
hence
\begin{equation}\label{mainlem5}
\bigcup_{k\in \Z}(B_r\cap \tilde G_k-2^{-k}(\overline z, Az_n))\subset 
\bigcup_{i=1}^{H}B_{ r}(x^{(i)}_z),
\end{equation}
where we have set  $H=K+4$ and $B_{ r}(x^{(K+j)}_z)=B_r-2^{-(h+j)}(\overline z, Az_n)$ for $j=-1, 0,1,2$.

Now we observe that for all $ k\in \Z$ and $ z\in B_1(0)$ with  $z_n>1/2$ we have
\begin{equation}\label{mainlem6}\tilde G_k-2^{-k}(\overline z,Az_n)\subseteq \tilde \Omega_k \,.
\end{equation}
Indeed, if $x\in \tilde G_k$, that is $2^{-k-2}<x_n-\varphi(\overline x)\le 2^{-k+1}$,  we have
\begin{multline*}
\varphi(\overline x-2^{-k}\overline z)-x_n+2^{-k}Az_n=\varphi(\overline x-2^{-k}\overline z)-\varphi(\overline x)+\varphi(\overline x)-x_n+2^{-k}Az_n\le\\
\le M2^{-k}|\overline z|-2^{-k-2}+ 2^{-k}Az_n\le M2^{-k}-2^{-k-2}+2^{-k}A<b2^{-k+1}\,,
\end{multline*}
and
\begin{multline*}
\varphi(\overline x-2^{-k}\overline z)-x_n+2^{-k}Az_n=\varphi(\overline x-2^{-k}\overline z)-\varphi(\overline x)+\varphi(\overline x)-x_n+2^{-k}Az_n>\\
>-M2^{-k}|\overline z|-2^{-k+1}+ 2^{-k}Az_n\ge -M2^{-k}-2^{-k+1}+ 2^{-k-1}A> 2^{-k-2}\,.
\end{multline*}

By (\ref{mainlem5}) and (\ref{mainlem6}) we deduce that 

\begin{equation*}\label{mainlem7}
B_r\cap \tilde G_k-2^{-k}(\overline z, Az_n) \subset \bigcup_{i=1}^{H}B_{ r}(x^{(i)}_z)\cap \tilde\Omega_k
\end{equation*}
which, combined with (\ref{mainlem3}), proves the validity of \eqref{mainlem1}.

We now prove statement (ii). By differentiating under integral sign, changing variables and integrating by parts, we get 
\begin{equation}\label{mainlem8}
D^{\alpha}f_k(x)=A^{-\alpha_n}2^{k|\alpha|}\int_{B_1(0)}f(\bar x-2^{-k}\bar z,x_n-A2^{-k}z_n )D^{\alpha}w(z)dz .
\end{equation}

We set $x^*=(\bar x,x_N-\frac{9}{4}A\rho_n(x))$, $\tilde x=(\bar x-2^{-k}\bar z,x_n-A2^{-k}z_n) $ and we denote by $V_{\tilde x}$ the conic body with vertex in $\tilde x$  constructed on the ball $B_{4\rho_n(x)}(x^*)$, i.e., $V_{\tilde x}=\cup_{y\in B_{4\rho_n(x)}(x^*)}(x^*,y)$ (where $(x^*,y)$ is the `open' segment joining $x^*$ and $y$).    Let $\mu \in C^{\infty}_c(B_1(0))$ be such that $\int_{B_1(0)}\mu dx=1$, and $w_x(y)=(4\rho_n(x))^{-n}\mu \left( \frac{x^*-y}{4\rho_n(x)}  \right)$.
By the Sobolev Integral Representation Formula (cf. \cite[Theorems 4, 5, Chap.~3]{b}), we get
\begin{equation}
\label{mainlem9}
f(\tilde x)=P(\tilde x,x)+\sum_{|\gamma|=l}r_{\gamma}(\tilde x, x),
\end{equation}
where 
\begin{equation*}
\label{mainlem10}
P(\tilde x,x)=\int_{B_{4\rho_n(x)}(x^*)}\sum_{|\gamma|\le l}\frac{1}{\gamma !}D^{\gamma}f(y)(\tilde x-y)^{\gamma}w_x(y)dy
\end{equation*}
and
 \begin{equation*}
%\label{mainlem11}
r_{\gamma}(\tilde x, x)=\int_{V_{\tilde x}}\frac{D^{\gamma}f(y)}{|\tilde x-y|^{n-l}}w_{\gamma , x}(y)dy,
\end{equation*}
where $w_{\gamma , x}$ is the appropriate kernel associated with $\omega $ appearing in the formula as in \cite[(3.38)]{b}. By \eqref{mainlem8} and 
\eqref{mainlem9} we deduce that 
\begin{multline*}\label{mainlem12}
D^{\alpha}f_{k}(x)=A^{-\alpha_n}2^{k|\alpha|}\int_{B_1(0)}P(\tilde x, x)D^{\alpha}\omega(z)dz \\
+
A^{-\alpha_n}2^{k|\alpha|}\int_{B_1(0)}\sum_{|\gamma|=l}r_{\gamma}(\tilde x , x)  D^{\alpha}\omega(z)dz .
\end{multline*}
We set 
\begin{equation*}
\label{mainlem13}
g_{\alpha}(x)=A^{-\alpha_n}2^{k|\alpha|}\int_{B_1(0)}P(\tilde x, x)D^{\alpha}\omega(z)dz .
\end{equation*}
We note that function $g_{\alpha}$ does not depend on $k$ (see \cite[p.~280]{b}). 

We now estimate 
\begin{multline}\label{mainlem14}
\| D^{\alpha}f_k-g_{\alpha}\|_{L^p(B_r\cap {\mathcal {U}}\cap \tilde G_k)}\\
=\biggl\| A^{-\alpha_n}2^{k|\alpha |}\sum_{|\gamma |=l}\int_{B_1(0)}r_{\gamma}(\tilde x, x)D^{\alpha }\omega (z)dz\biggr\|_{L^p(B_r\cap {\mathcal {U}}\cap \tilde G_k)}.
\end{multline}
To do so, we proceed as follows
%\begin{multline}
%\label{mainlem15}
%\biggl\| \int_{B_1(0)}  r_{\gamma}(\tilde x, x)D^{\alpha }\omega (z)dz\biggr\|_{L^p(B_r\cap \tilde G_k)}\\
%\le  c
 %\int_{B_1(0)} \| r_{\gamma}(\tilde x, x)\|_{L^p(B_r\cap \tilde G_k)}dz\\
%= c \int_{B_1(0)}\biggl\| \int_{V_{\tilde x}}\frac{D^{\gamma }f(y)\chi_{\tilde \Omega_k}(y) }{|\tilde x-y|^{n-l}}w_{\gamma ,x}
%(y)dy\biggr\|_{L^p(B_r\cap \tilde G_k)}dz\\
%\le c \int_{B_1(0)}\biggl\| \int_{B_{20 A2^{-k}}(0)}\frac{D^{\gamma }f(\tilde x-\eta)\chi_{\tilde \Omega_k}(\tilde x-\eta) }{| \eta|%^{n-l}}d\eta\biggr\|_{L^p(B_r\cap \tilde G_k)}dz\\
%\le c \int_{B_1(0)} \int_{B_{20 A2^{-k}}(0)}\frac{ \|   D^{\gamma }f(\tilde x-\eta)\chi_{\tilde \Omega_k}(\tilde x-\eta)  \|%_{L^p(B_r\cap \tilde G_k)}  }{| \eta|^{n-l}}d\eta dz
%\\
%\le c \int_{B_1(0)} \int_{B_{1}(0)}\frac{ \|   D^{\gamma }f(\tilde x-20A2^{-k}\xi)\chi_{\tilde \Omega_k}(\tilde x-20A2^{-k}\xi)  \|%_{L^p(B_r\cap \tilde G_k)}  }{2^{kl}| \xi|^{n-l}}d\xi dz   \\
%\le c \int_{B_1(0)} \int_{B_{1}(0)}\frac{ \|   D^{\gamma }f\chi_{\tilde \Omega_k}  \|_{L^p(B_r\cap \tilde G_k  -2^{-k}(\bar z, %Az_n)-20A2^{-k}\xi  )   }  }{2^{kl}| \xi|^{n-l}}d\xi dz
%\end{multline}
\begin{multline*}
%\label{mainlem15bis}
\biggl\| \int_{B_1(0)}  r_{\gamma}(\tilde x, x)D^{\alpha }\omega (z)dz\biggr\|^p_{L^p(B_r\cap {\mathcal {U}}\cap \tilde G_k)}\\ \le c\int_{B_r\cap{\mathcal {U}}\cap \tilde G_k}\left|\int_{B_1(0)}r_{\gamma}(\tilde x, x)  dz\right|^pdx\\ \le  c\int_{B_r\cap {\mathcal {U}}\cap \tilde G_k}\int_{B_1(0)}|r_{\gamma}(\tilde x, x)  |^pdz dx\,.
\end{multline*} 
Using Minkowskii inequality and the fact that $V_x\subset \tilde\Omega_k$ and that ${\rm diam }V_x\le 20A2^{-k}$ (cf. \cite[(6.83) and pp.~278-279]{b}), we get 
\begin{multline*}
\left(\int_{B_1(0)}|r_{\gamma}(\tilde x, x)  |^pdz\right)^{1/p}\\
=\left(\int_{B_1(0)}\left|  \int_{V_{\tilde x}}\frac{D^{\gamma }f(y)\chi_{\tilde \Omega_k}(y) }{|\tilde x-y|^{n-l}}w_{\gamma ,x}(y)dy      \right|^pdz\right)^{1/p}\\
 \le c\left(\int_{B_1(0)}\left|  \int_{B_{20A2^{-k}}(0)}\frac{D^{\gamma }f(\tilde x-\eta )\chi_{\tilde \Omega_k}( \tilde x-\eta ) }{|\eta |^{n-l}}d\eta       \right|^pdz\right)^{1/p}
\\
\le c\int_{B_{20A2^{-k}}(0)} \left( \int_{B_1(0)}\frac{|D^{\gamma }f(\tilde x-\eta )\chi_{\tilde \Omega_k}( \tilde x-\eta )|^p }{|\eta |^{p(n-l)}}  dz \right)^{1/p}      d\eta\\
=c\int_{B_{20A2^{-k}}(0)} ( \int_{B_1(0)}|D^{\gamma }f(\tilde x-\eta )\chi_{\tilde \Omega_k}( \tilde x-\eta )|^p       dz )^{1/p}      \frac{d\eta  }{ |\eta |^{n-l}}\\
=c\int_{B_{20A2^{-k}}(0)} \left( \int_{B_1(0) -\eta}|D^{\gamma }f(\tilde x)\chi_{\tilde \Omega_k}( \tilde x )|^p       dz \right)^{1/p}      \frac{d\eta  }{ |\eta |^{n-l}}\\
\le c\int_{B_{20A2^{-k}}(0)}\left( \int_{B_{1+20A 2^{-k}}(0) }|D^{\gamma }f(\tilde x)\chi_{\tilde \Omega_k}( \tilde x )|^p       dz \right)^{1/p}      \frac{d\eta  }{ |\eta |^{n-l}}\\
\le c \left( \int_{B_{1+20A 2^{-k}}(0) }|D^{\gamma }f(\tilde x)\chi_{\tilde \Omega_k}( \tilde x )|^p       dz \right)^{1/p}    \int_{B_{20A2^{-k}}(0)}  \frac{d\eta  }{ |\eta |^{n-l}}\\
\le c 2^{-kl} \left( \int_{B_{1+20 A2^{-k}}(0) }|D^{\gamma }f(\tilde x)\chi_{\tilde \Omega_k}( \tilde x )|^p       dz\right )^{1/p}\,,    \\
\end{multline*}
from which it follows that 
\begin{multline}
\label{mainlem15}
\biggl\| \int_{B_1(0)}  r_{\gamma}(\tilde x, x)D^{\alpha }\omega (z)dz\biggr\|^p_{L^p(B_r\cap {\mathcal {U}}  \cap \tilde G_k)}\\
 \le c 2^{-klp}\int_{B_r\cap {\mathcal {U}}\cap\tilde G_k}  \int_{B_{1+20 A2^{-k}}(0) }|D^{\gamma }f(\tilde x)\chi_{\tilde \Omega_k}( \tilde x )|^p       dz dx   \\
\le c 2^{-klp} \int_{B_{1+20A 2^{-k}}(0) }\int_{B_r\cap {\mathcal {U}}\cap\tilde G_k} |D^{\gamma }f(\tilde x)\chi_{\tilde \Omega_k}( \tilde x )|^p  dx     dz\,.  
\end{multline}
By \eqref{mainlem15} and \eqref{mainlem14} we obtain

\begin{multline*}%\label{mainlem16bis}
\| D^{\alpha}f_k-g_{\alpha}\|^p_{L^p(B_r\cap {\mathcal {U}}\cap\tilde G_k)}\\
\le c 2^{pk(|\alpha |-l)} \sum_{|\gamma|=l}\int_{B_{1+20A 2^{-k}}(0) }\int_{B_r\cap {\mathcal {U}}\cap\tilde G_k} |D^{\gamma }f(\tilde x)\chi_{\tilde \Omega_k}( \tilde x )|^p  dx     dz\,.  
\end{multline*}

If we denote by $\bar k$ the minimum $k\in \Z$ such that $B_r\cap  {\mathcal {U}}\cap \tilde G_k\ne \emptyset$, then we easily see that $20A2^{-{\bar k}}\le c d$ and

\begin{multline}\label{mainlem17bis}
\| D^{\alpha}f_k-g_{\alpha}\|^p_{L^p(B_r\cap {\mathcal {U}} \cap \tilde G_k)}\\
\le c2^{pk(|\alpha |-l)} \sum_{|\gamma|=l}\int_{B_{1+20A 2^{-\bar k}}(0) }\int_{B_r\cap {\mathcal {U}}\cap\tilde G_k} |D^{\gamma }f(\tilde x)\chi_{\tilde \Omega_k}( \tilde x )|^p  dx     dz \\
\le c2^{pk(|\alpha |-l)} \sum_{|\gamma|=l}\int_{B_{1+cd}(0) } \int_{B_r\cap {\mathcal {U}}\cap\tilde G_k} |D^{\gamma }f(\tilde x)\chi_{\tilde \Omega_k}( \tilde x )|^p  dx dz\, . 
\end{multline}

%A questo punto, procederemo come al solito con la seguente differenza: l'integrale su $\eta $ non c'\`{e} pi\`{u} e l'integrale su z %va fatto sulla palla
%$B_{1+20 2^{-\bar k}}(0) $ invece che su $B_1(0)$. Ovvero: ci sbarazziamo dell'integrale `singolare' in $\eta$ ma paghiamo la %grandezza di $z$. Questo si ripercuote solo 
%sulle 4 palle che dipendono da z, che saranno di raggio $ [2+(M+1)(1+20 2^{-\bar k})]r$ (vedi il lemma sopra dove quantifichiamo %il raggio). 

Now by Lemma \ref{inc} and its proof we obtain that  for all $z\in B_{1+cd}(0)$
\begin{multline}\label{mainlem16}
\bigcup_{k\in \Z}(B_r\cap \tilde G_k-2^{-k}(\overline z, Az_n) )=\\
\bigcup_{k=h-1}^{h+2}(B_r\cap \tilde G_k-2^{-k}(\overline z, Az_n) )\cup 
\bigcup_{k=h+3}^\infty(B_r\cap \tilde G_k-2^{-k}(\overline z, Az_n) )\\
\subset   \bigcup_{k=h-1}^{h+2}(B_r\cap \tilde G_k-2^{-k}(\overline z, Az_n) )\cup
B_{S^\prime r}(x_{z }) ,
\end{multline}
where $S^\prime=2+(M+1)(1+cd)$ and $B_{S^\prime r}(x_{z })$ is the ball provided by Lemma \ref{inc}.  

Clearly, as in the proof of Lemma \ref{inc}, one can easily deduce by (\ref{mainlem16})
the existence of $H_{\mathcal{U}}\in \N$ and of $H_{\mathcal{U}}$ balls 
$B_{r}(x_z^{(i)})$, $i=1,\dots , H_{\mathcal{U}}$, defined for all $z\in B_{1+cd}(0)$ as in the statement such that

\begin{equation}\label{mainlem16bis}
\bigcup_{k\in \Z}(B_r\cap \tilde G_k-2^{-k}(\overline z, Az_n) )=
\bigcup_{i=1}^{H_{\mathcal{U}}} B_{r}(x_z^{(i)}) .\end{equation}

By \eqref{mainlem17bis} and (\ref{mainlem16bis}) we get 
\begin{multline*}\label{mainlem18bis}
\| D^{\alpha}f_k-g_{\alpha}\|^p_{L^p(B_r\cap {\mathcal{U}}\cap\tilde G_k)}\\
\le c2^{pk(|\alpha |-l)} \sum_{|\gamma|=l}\int_{B_{1+cd}(0) }\| D^\gamma f\chi_ {\tilde \Omega_k}\|^p_{L^{p}(\cup_{i=1}^{H_{\mathcal{U}}}   B_{ r}(x^{(i)}_z)\cap \tilde \Omega_k) }    dz  \\
\le c2^{pk(|\alpha |-l)} \int_{B_{1+cd}(0) } \sum_{|\beta|=l}  \| D^{\beta} f\|^p_{L^{p}(\cup_{i=1}^{H_{\mathcal{U}}}   B_{ r}(x^{(i)}_z)\cap \tilde \Omega_k) }    dz.  
\end{multline*}

 %Infine si osserva che $2^{ -\bar k}\le c d$ dove $d$ \`{e} la distanza della 
%palla da $ \Omega $.
%In definitiva potremmo scrivere una cosa del tipo
%\begin{multline}\label{mainlem19bis}
%\| D^{\alpha}f_k-g_{\alpha}\|^p_{L^p(B_r\cap \tilde G_k)}\\
%\le c2^{pk(|\alpha |-l)} \int_{B_{1+cd}(0) }\| f\|^p_{w^{l,p}(\cup_{i=1}^4   B_{S_i(1+d) r}(x^{(i)}_z\cap \tilde \Omega_k) }    %dz  
%\end{multline}
%con $S_i$ costanti assolute. Tutto cio' comporta la comparsa della distanza della palla dal bordo nella nostra stima finale (con una %potenza positiva). Ma questa cosa 
%sembra da una parte naturale, dall'altra non problematica a causa del fatto che le estensioni di fatto sono a `corto raggio d'azione'. 

%By proceeding as above, we have that

%where $B_{S_i r}(x^{(i)}_{z, \xi }))$ are balls as in the statement. Thus, 
 %\begin{multline}
%\label{mainlem17}
%\biggl\| \int_{B_1(0)}  r_{\gamma}(\tilde x, x)D^{\alpha }\omega (z)dz\biggr\|_{L^p(B_r\cap \tilde G_k)}\\
%\le c \int_{B_1(0)} \int_{B_{1}(0)}\frac{ \|   D^{\gamma }f\chi_{\tilde \Omega_k}  \|_{L^p(  \cup_{i=1}^{4}B_{S_i r}(x^{(i)}_{z,%\xi })    )   }  }{2^{kl}| \xi|^{n-l}}d\xi dz
%\\
%\le c \int_{B_1(0)} \int_{B_{1}(0)}\frac{ \|   D^{\gamma }f  \|_{L^p(  \cup_{i=1}^{4}B_{S_i r}(x^{(i)}_{z,\xi })\cap %\tilde\Omega_k    )   }  }{2^{kl}| \xi|^{n-l}}d\xi dz
%\end{multline}
%By (\ref{mainlem14}) and (\ref{mainlem17}) we deduce the validity of statement (ii). 
\end{proof}

Then we have the following

\begin{theorem}\label{teoball}
Let $B_r$ a ball in $\R^n$  with radius $r$ such that $B_r \cap G\ne \emptyset$. The following statements hold:
\begin{itemize}
\item[(i)] There exist $c>0$, $H\in \N$ and  $H$ balls  $B_{ r}(x^{(i)}_z)$,  $i=1,\dots,H$,  as in Lemma~\ref{mainlem} such that  
\begin{equation*}\label{teoball1}
\| Tf \|_{L^p(B_r \cap G)}^p\le c\int_{\{z\in B_1(0):\, z_n>1/2\}} \| f\|^p_{L^p(\cup_{i=1}^H   B_{ r}(x^{(i)}_z) \cap \Omega) }dz\, ,
\end{equation*} 
for all $f\in W^{l,p}(\Omega)$. 
\item[(ii)]
Let ${\mathcal{U}}\subset \R^n$ be a  fixed measurable set and let $d=\sup\{\rho_n(x):\ x\in B_r\cap {\mathcal{U}} \}$.  
Assume that $d<\infty$.
There exist  $c>0$,  $H_{{\mathcal{U}}}\in \N$  and $H_{{\mathcal{U}}}$   balls  $B_{ r}(x^{(i)}_{z})$,  $i=1,\dots,H _{{\mathcal{U}}} $,  such that  
for all $\alpha\in \N_0^n$ with $|\alpha |   =  l$
\begin{multline}\label{teoball2}
\|  D^{\alpha}  T f \|^p_{L^p(B_r\cap {\mathcal{U}}\cap G  )}
\le   
 c \int_{B_{1+cd}(0)}\sum_{|\beta|=l}   \| D^{\beta}f\|^p_{L^{p}(\cup_{i=1}^{H_{\mathcal{U}}}   B_{ r}(x^{(i)}_z)\cap  \Omega) }  dz \, ,
\end{multline}
for all $f\in W^{l,p}(\Omega)$.
\end{itemize}
\end{theorem}

\begin{proof}

We begin with statement (i). Since  the multiplicity of the covering $\{{\rm supp} \psi_k\}_{k\in \Z}$ is equal to $2$, by \eqref{mainlem1} with $\alpha =0$ and H\"{o}lder inequality, we get

\begin{multline*}%\label{teoball3}
\|Tf\|^p_{L^p(B_r\cap G)}=\int_{B_r\cap G}\left|\sum_{k=-\infty}^{+\infty}\psi_k(x)f_k(x) \right|^pdx\le\\
\le 2^{p-1}\int_{B_r\cap G}\sum_{k=-\infty}^{+\infty}|\psi_k(x)f_k(x)|^pdx 
=2^{p-1}\sum_{k=-\infty}^{+\infty}\int_{B_r\cap G}|\psi_k(x)f_k(x)|^pdx \\
\le 2^{p-1}\sum_{k=-\infty}^{+\infty}\int_{B_r\cap \tilde G_k}|\psi_k(x)f_k(x)|^pdx\le 2^{p-1}\sum_{k=-\infty}^{+\infty}\int_{B_r\cap \tilde G_k}|f_k(x)|^pdx\\
\le c\int_{\{z\in B_1(0):\, z_n>1/2\}} \sum_{k=-\infty}^{+\infty} \| f\|^p_{L^p(\cup_{i=1}^H B_{ r}(x^{(i)}_z) \cap \tilde \Omega_k )}dz\\
\le c \kappa_{ \Omega }\int _{\{z\in B_1(0):\, z_n>1/2\}}\| f\|^p_{L^p(\cup_{i=1}^H   B_{ r}(x^{(i)}_z) \cap  \Omega )}dz,
\end{multline*}
for all $f\in W^{l,p}(\Omega)$, where $\kappa_{\Omega}$  denotes the multiplicity of the covering $\{\tilde\Omega_k\}_{k\in \Z}$ of $\Omega$. Thus, in order to conclude it suffices to observe that        $\kappa_{\Omega}$ depends only on $M$ by \cite[Remark~12, Chap.~6]{b}. \\

We now prove statement (ii). Using statements (i) and (ii) in Lemma~\ref{mainlem}, statement (iv) in Lemma~\ref{partition},  and the simple equality
$\sum_{k\in \Z}D^{\alpha -\beta}\psi_k D^{\beta }f_k=\sum_{k\in \Z}D^{\alpha -\beta}\psi_k (D^{\beta }f_k -g_{\beta})$  for $\beta <\alpha$   (recall that $g_{\beta}$ does not depend on $k$)
we get
\begin{multline*}\label{mainlem4}
\| D^{\alpha} T f \|^p_{L^{p}(B_r\cap {\mathcal{U}} \cap G  )}\\
=\int_{B_r\cap {\mathcal{U}}\cap G }\Big|\sum_{\beta \le \alpha }\frac{\alpha!}{(\alpha -\beta )!\beta!}\sum_{k\in \Z}D^{\alpha -\beta}\psi_k D^{\beta }f_k   \Big|^pdx\\
\le c\int_{B_r\cap {\mathcal{U}}\cap G }\Big|\sum_{k\in \Z}\psi_k D^{\alpha }f_k    \Big|^pdx\\
+
c\int_{B_r\cap {\mathcal{U}}\cap G }\Big|\sum_{\beta < \alpha }\frac{\alpha!}{(\alpha -\beta )!\beta!}\sum_{k\in \Z}D^{\alpha -\beta}\psi_k D^{\beta }f_k    \Big|^pdx\\
\le c   \sum_{k\in \Z}\| D^{\alpha }f_k \|^p_{L^p(B_r\cap {\mathcal{U}}\cap \tilde G_k )}   + c  \sum_{\beta < \alpha ,   k\in \Z }    \| D^{\alpha -\beta}\psi_k( D^{\beta }f_k-g_{\beta}  )  \|^p_{L^p(B_r\cap {\mathcal{U}}\cap \tilde G_k )}\\ 
  \end{multline*}
\begin{multline*}
\le c\int_{B_1(0)} \sum_{k=-\infty}^{+\infty} \| D^{\alpha }f\|^p_{L^p(\cup_{i=1}^H  B_{ r}(x^{(i)}_z) \cap \tilde \Omega_k )}dz\\
+c \sum_{\beta < \alpha ,   k\in \Z }   2^{k|\alpha -\beta|p} \| D^{\beta }f_k-g_{\beta}    \|^p_{L^p(B_r\cap {\mathcal{U}}\cap \tilde G_k )}\\
\le c \kappa_{\Omega }\int _{B_1(0)}\| D^{\alpha }f\|^p_{L^p(\cup_{i=1}^{H_{\mathcal{U}}}  B_{ r}(x^{(i)}_z) \cap \tilde \Omega )}dz\\
+ c \sum_{\beta < \alpha ,   k\in \Z }   2^{pk|\alpha -\beta|}2^{pk(|\beta|-l)}\int_{B_{1+cd}(0)}  \sum_{|\beta|=l}\|D^{\beta} f\|^p_{L^{p}(\cup_{i=1}^{H_{\mathcal{U}}} B_{ r}(x^{(i)}_z)\cap \tilde \Omega_k ) }  dz\\
\le c \kappa_{\Omega }\int _{B_1(0)}\| D^{\alpha }f\|^p_{L^p(\cup_{i=1}^{H_{\mathcal{U}}}   B_{ r}(x^{(i)}_z) \cap \Omega )}dz\\
+c \kappa_{\Omega }\int_{B_{1+cd}(0)} \sum_{|\beta|=l}  \| D^{\beta}f\|^p_{L^{p}(\cup_{i=1}^{H_{\mathcal{U}}}   B_{ r}(x^{(i)}_z)\cap  \Omega) }  dz,
\end{multline*}
which allows to conclude. 
\end{proof}

\begin{cor}\label{morrey} Let $l\in \N_0$,  $p\in [1, \infty [$,   $\phi:]0,+\infty[\to ]0,+\infty[$. The following statements hold:
\begin{itemize}
\item[(i)] There exists $c>0$  such that 
\begin{equation}\label{morrey1} \|Tf\|_{M^{\phi  , \delta }_p(\R^n )}\le c\|f\|_{M^{\phi, \delta}_p(\Omega )} , \end{equation}
for all $\delta \in ]0,\infty ]$ and $f\in W^{l,p}(\Omega)$. 
\item[(ii)]   Let $D>0$. There exists $c>0$ such that 
\begin{equation}\label{morrey2} \|D^{\alpha }Tf\|_{M^{\phi ,\delta }_p(\R^n )}\le  c \sum_{|\beta |=l}\|D^{\beta}f\|_{M^{\phi , \delta }_p(\Omega )}  ,\end{equation}
for all $\delta \in ]0,\infty ]$, $f\in W^{l,p}(\Omega)$ with ${\rm supp}f\subset \{x\in \Omega :\ |\rho_n(x)|<D \}$, and $\alpha \in \N_0^n$  with $|\alpha|= l$.
\end{itemize}

\end{cor}

\begin{proof}Let $\delta \in ]0,\infty]$ be fixed and let $B_r$ be a ball in $\R^n$ with radius $r<\delta $. 
We begin with proving statement (i). Then by Theorem~\ref{teoball} (i) we have
\begin{multline*}%\label{morrey3}
\| Tf \|_{L^p(B_r )}^p=\| f \|_{L^p(B_r \cap \Omega)}^p+ \| Tf \|_{L^p(B_r \cap G)}^p   \\
 \le \phi(r)\| f \|^p_{M^{\phi, \delta }_p(\Omega)} +
  c\int_{\{z\in B_1(0):\, z_n>1/2\}} \| f\|^p_{L^p(\cup_{i=1}^H   B_{ r}(x^{(i)}_z) \cap \Omega) }dz \\
 \le \phi(r)\| f \|^p_{M^{\phi , \delta }_p(\Omega)} +
  c   \phi(r) \int_{\{z\in B_1(0):\, z_n>1/2\}}\| f \|^p_{M^{\phi ,\delta }_p(\Omega)}  dz\,,
\end{multline*} 
for all  $f\in W^{l,p}(\Omega)$, which provides the validity of \eqref{morrey1}.

We now prove statement (ii). Let $f$ be as in statement (ii). We claim that $Tf(x)=0$ for all $x\in G$ with $\rho_n(x)>8D$. Indeed, assume that    $x\in G$ with $\rho_n(x)>8D$ and  let $k\in \Z$ be such that $x\in \tilde G_k$.  By \cite[Remark~11 and (6.83)]{b}, it follows that the value of $f_k(x)$ depends only
on the values of $f_{|\tilde\Omega_k}$ and in particular, if $f_{|\tilde\Omega_k}=0$ then $f_k(x)=0$.  Since $\rho_n(x)>8D$ and $x\in  \tilde G_k$, we have that $2^{-k+1}>8D$, hence $2^{-k-2}>D$ which implies that 
$\tilde\Omega _k\subset \{ y\in\Omega : |\rho_n(y)|>D   \}\subset ({\rm supp}f)^c$. 
Thus, $f_k(x)=0$ as we have claimed. 
We set
$$
G_D=\{x\in  G:\ \rho_n(x)\le 8D\}. 
$$
Since ${\rm supp} Tf\subset G_D$, we have

\begin{multline}\label{morrey4}
\| D^{\alpha }Tf \|_{L^p(B_r )}^p=\|D^{\alpha } f \|_{L^p(B_r \cap \Omega)}^p+ \| D^{\alpha }Tf \|_{L^p(B_r \cap G)}^p   \\
= \|D^{\alpha } f \|_{L^p(B_r \cap \Omega)}^p+ \| D^{\alpha }Tf \|_{L^p(B_r \cap G_D  )}^p \, .\\ 
\end{multline} 
   
By applying (\ref{teoball2}) with ${\mathcal{U}}=G_D$ and observing that $d=\sup\{\rho_n(x):\ x\in B_r\cap {\mathcal{U}}\}\le 8D$, by (\ref{morrey4}) we get

\begin{multline*}\label{morrey4bis}
\| D^{\alpha }Tf \|_{L^p(B_r )}^p\\
\le c  \phi(r)\| D^{\alpha }f \|^p_{M^{\phi ,\delta }_p(\Omega)} 
+ c \int_{B_{1+cd}(0)}\sum_{|\beta |=l}   \| D^{\beta} f\|^p_{L^{p}(\cup_{i=1}^{H_{\mathcal{U}}}   B_{ r}(x^{(i)}_z)\cap  \Omega) }  dz\\
\le  c  \phi(r)\| D^{\alpha }f \|^p_{M^{\phi ,\delta }_p(\Omega)} +
c  \phi(r)\int_{B_{1+8cD}(0)}  \sum_{|\beta|=l}  \|D^{\beta} f\|^p_{M^{\phi , \delta }_p(\Omega ) }dz \\
\end{multline*} 
which implies \eqref{morrey2}.

\end{proof}

\subsection{The case of bounded Lipschitz subgraphs }

In this subsection, we  consider  bounded elementary domains with Lipschitz boundaries. Namely, these domains are bounded Lipschitz subgraphs. However, in order to treat the case
of general Lipschitz domains in the next section, we need to take into account a number of parameters describing the size of such subgraphs. For this reason, following \cite{b}, it is convenient to give the following definition.

\begin{defi}\label{ele} Let $d, D >0$, $M\geq 0$. We say an open set ${\mathcal{H}}$ in $\R^n$ is a bounded elementary domain with Lipschitz boundary and parameters
$d, D$, $M$ if ${\mathcal{H}}$ can be represented as 
\begin{equation}\label{ele1}
{\mathcal{H}}=\left\{x\in \R^n:\ \bar x\in W,\ a_n<x_n<\varphi (\bar x) \right\}
\end{equation}
where  $W=\Pi_{i=1}^{n-1}]a_i,b_i[$, $-\infty <a_i<b_i<\infty $ for all $i=1,\dots ,n$, 
${\rm diam }{\mathcal{H}}<D$ and $\varphi :W\to \R$ is a   Lipschitz function such that
$$
a_n+d<\varphi\ \ {\rm and}\ \ {\rm Lip }\varphi \le M. 
$$ 
\end{defi}

We note that if $\varphi$ is a Lipschitz function  as in Definition~\ref{ele}, then $\varphi $ can be extended to the whole of $\R^{n-1}$ 
by means of  a Lipschitz function $F_{\varphi}$ such that ${\rm Lip}F_{\varphi}={\rm Lip}\varphi $. In particular, given an elementary domain ${\mathcal{H}}$ represented as 
in (\ref{ele1}), we can define the following open set
\begin{equation*}\label{ele2}
\Omega_{\mathcal{H}}=\{x\in \R^n:\ x_n< F_{\varphi }(\bar x) \}.
\end{equation*}

We find it convenient to set
$$
\widetilde W^{l,p}({\mathcal{H}})=\left\{f\in W^{l,p}({\mathcal{H}}) :\ {\rm supp }f  \subset \Pi_{i=1}^{n}]a_i,b_i[  \right\}.
$$

Given a function $f\in \widetilde W^{l,p}({\mathcal{H}} )$  then the extension-by-zero $f_0$ of $f$ (defined by
$f_0(x)=f(x)$ for all $x\in {\mathcal{H}}$ and $f_0(x)=0$ for all $x\in \Omega_{\mathcal{H}}$) belongs to $W^{l,p}(\Omega_{\mathcal{H}})$, because the distance of ${\rm supp} f$ from the boundary of  $\Pi_{i=1}^{n}]a_i,b_i[$ is positive, hence a standard truncation argument is applicable. Consider now the extension operator $T$ 
defined by (\ref{burext}) for the open set $\Omega=\Omega_{\mathcal{H}}$. For all functions $f\in \widetilde W^{l,p}({\mathcal{H}} )$ we set
\begin{equation}\label{th}
T_{\mathcal{H}}f=Tf_0. 
\end{equation}

It is clear that $T_{\mathcal{H}}f\in W^{l,p}(\R^N)$ for all  $f\in \widetilde W^{l,p}({\mathcal{H}} )$. The following theorem is an easy consequence of Corollary~\ref{morrey}.

\begin{theorem}\label{morreyele}  Let $l\in \N_0$,  $p\in [1,\infty [$  and $\phi:]0,+\infty[\to ]0,+\infty[$.
 Let ${\mathcal{H}} $ be a bounded elementary domain with Lipschitz boundary and parameters
$ d, D$, $M$. Then there exists $c>0$ depending only on $n, l,p, D, M$ such that  
\begin{equation*}\label{morrey1ele} \|T_{\mathcal{H}}f\|_{M^{\phi ,\delta }_p(\R^n )}\le c\|f\|_{M^{\phi ,\delta }_p({\mathcal{H}})},\end{equation*}
and 
\begin{equation}\label{morrey2ele} \|D^{\alpha }T_{\mathcal{H}}f\|_{M^{\phi ,\delta }_p(\R^n )}\le       c\sum_{|\beta |=l}\|D^{\beta}f\|_{M^{\phi , \delta }_p({\mathcal{H}} )},\end{equation}
for all $\delta \in ]0, \infty ]$, $f\in \widetilde W^{l,p}({\mathcal{H}})$ and  $\alpha \in \N_0^n$  with $|\alpha|= l$.
\end{theorem}

\begin{proof} The proof immediately follows by Corollary~\ref{morrey} and by observing that for all $f\in \widetilde W^{l,p}(H)$
we have that ${\rm supp }f_0\subset \{x\in \Omega_H:\ |\rho_n(x)|<D \}$.
\end{proof}

\section{Burenkov's Extension Operator on general Lipschitz open sets}

We recall the definition of open set with Lipschitz boundary. Here and in the sequel, given a set $C$ in $\R^n$ and $d>0$ we denote by $C_d$ the set
$\{x\in C:{\rm dist} (x,\partial C)>d\}$.

\begin{defi} Let $d >0$, $M\geq 0$, $s\in \N \cup \{\infty \}$. Let $\{V_{j}\}_{j=1}^s$ be a family of cuboids, i.e. for every $j=\overline{1,s}$ there exists 
an isometry $\lambda_j$ in $\R^n$ such that 
$$
\lambda_j( V_j  )= \Pi_{i=1}^n]a_{i,j}, b_{i,j}[
$$ 
where $0<a_{i,j}<a_{i,j}+d<b_{i,j}$. Assume that $D:=\sup_{j=\overline{1,s}}{\rm diam }V_j< \infty $, $(V_j)_{d}\ne \emptyset $ for all $j=\overline{1,s}$,   and that the multiplicity of the covering $\{V_{j}\}_{j=1}^s$ is finite. 
We then say that ${\mathcal{A} }=(s,d, \{V_{j}\}_{j=1}^s, \{\lambda_{j}\}_{j=1}^s )$ is an atlas.  

Let $M\geq 0$. We say that an open set $\Omega$ in $\R^n$ is of class $C^{0,1}_M({\mathcal{A}})$  if the following conditions are satisfied:\\
 
(i) For every  $j=\overline{1,s}$, we have $\Omega \cap (V_j)_d\ne \emptyset$. \\

(ii) $\Omega \subset \cup_{j=1}^{s}(V_j)_d$.\\

(iii) For every  $j=\overline{1,s}$, the set ${\mathcal{H}}_j:=\lambda_j(\Omega \cap V_j)$  satisfies the following condition: either ${\mathcal{H}}_j= \Pi_{i=1}^n]a_{i,j}, b_{i,j}[ $ (in which case $V_j\subset \Omega $),  or ${\mathcal{H}}_j$ is a bounded  elementary domain 
 with Lipschitz boundary and parameters
$d, D$, $M$ of the form
\begin{equation*}\label{ele1bis}
{\mathcal{H}}_j=\left\{x\in \R^n:\ \bar x\in W_j,\ a_{n,j}<x_n<\varphi_j (\bar x) \right\}
\end{equation*}
where  $\varphi_j $ is a real-valued  Lipschitz function defined on $W_j= \Pi_{i=1}^{n-1}]a_{i,j}, b_{i,j}[$ such that
$$
a_{n,j}+d<\varphi_j\ \ {\rm and}\ \ {\rm Lip }\varphi_j \le M 
$$ 
(in which case $V_j\cap \partial \Omega \ne \emptyset$).
\end{defi}

Let ${\mathcal{A}}=(s,d, \{V_{j}\}_{j=1}^s, \{\lambda_{j}\}_{j=1}^s )$ be an atlas and let ${\mathcal{H}}_j=\lambda_j(\Omega \cap V_j)$ for all $j=\overline{1,s}$,  as above.  For every $j=\overline{1,s}$, we consider an extension operator
$T_{{\mathcal{H}}_j}$  from $\widetilde W^{l,p}({\mathcal{H}}_j)$ to $W^{l,p}(\R^n)$ which is the operator defined by (\ref{th}) if $V_j\cap \partial\Omega \ne \emptyset$ and 
is just the extension-by-zero operator if $V_j\subset \Omega $. 

Next, for every $j=\overline{1,s}$, we consider the push-forward 
operator $\Lambda_j$  from $W^{l,p}(\Omega \cap V_j)$ to $W^{l,p}({\mathcal{H}}_j)$ defined by $\Lambda_jf=f\circ \lambda_j^{(-1)}$ for all $f\in W^{l,p}(\Omega \cap V_j)$   and we set $\widetilde W^{l,p}(\Omega \cap V_j) = \Lambda_j^{(-1)}(\widetilde W^{l,p}({\mathcal{H}}_j)    ) $. Note that $\widetilde W^{l,p}(\Omega \cap V_j)$  is the space of functions in $W^{l,p}(\Omega \cap V_j)$ such that their support has positive distance from the boundary of $V_j$.    Moreover, we  consider the corresponding pull-back operator defined now from $W^{l,p}(\R^n)$ to itself, which we call directly $\Lambda_j^{(-1)}$ and which is defined by  $\Lambda_j^{(-1)}u=u\circ \lambda_j$ for all $u\in 
 W^{l,p}(\R^n)$.  Finaly, we set
\begin{equation*}
%\label{burextgen}
T_j:=\Lambda_j^{(-1)}\circ T_{{\mathcal{H}}_j}\circ \Lambda_j,
\end{equation*}
and we note that $T_j$ is a well-defined linear continuous extension operator from $\widetilde W^{l,p}(\Omega \cap V_j)$ to $ W^{l,p}(\R^n)$.

Following \cite[p.265]{b}, given an open set $\Omega $ of class $C^{0,1}_M({\mathcal{A}})$, we consider a family of   functions $\{\psi\}_{j=1}^s$  such that $\psi_j\in C^{\infty }_c(\R^n)$, $ {\rm supp} \psi_j\subset  (V_j)_{d} $, $0\le \psi_j\le 1$, $\sum_{j=1}^s\psi_j^2(x)=1$ for all $x\in \Omega$ and such that $\| D^{\alpha }\psi_j\|_{L^{\infty }(\R^n)}$ $\le M$ for all $j=\overline{1,s}$ and $\alpha\in \N_0^n$ with $|\alpha |\le l$, where $M$ depends only on $n,l,d$.   

We are able now to define the Burenkov's Extension Operator $T$ from $W^{l,p}(\Omega )$ to $W^{l,p}(\R^n)$ as follows:

\begin{equation}
\label{burextgen}
Tf= \sum_{j=1}^s\psi_j  T_j(f\psi_j),
\end{equation}
for all $f\in W^{l,p}(\Omega )$.  Note that ${\rm supp }\Lambda_j( f\psi_j  )\subset \Pi_{i=1}^s]a_j,b_j[$, hence $T_j (f\psi_j)$ is well-defined.

Before giving the proof of the main result of this section, we  need to prove the  following lemma.

\begin{lemma}\label{lemmarot}
Let $l\in \N_0$,  $p\in [1,\infty [$  and $\phi:]0,+\infty[\to ]0,+\infty[$. Let ${\mathcal{A}}$  be an atlas in $\R^n$, $M\geq 0$ and $\Omega$ be an open set  of class $C^{0,1}_M({\mathcal{A}})$.  Then there exists  $c>0$ depending only on $n, {\mathcal{A}}, M, p$ such that 
\begin{equation*}\label{morrey1elebis} \|T_jf\|_{M^{\phi ,\delta }_p(\R^n )}\le c\|f\|_{M^{\phi ,\delta}_p(\Omega \cap V_j )  },\end{equation*}
and 
\begin{equation}\label{morrey2elebis} \|D^{\alpha }T_jf\|_{M^{\phi ,\delta }_p(\R^n )}\le c \sum_{|\beta |= l}\|D^{\beta}f\|_{M^{\phi ,\delta }_p(\Omega \cap V_j )},\end{equation}
for all $\delta \in ]0,\infty]$, $j=\overline{1,s}$,  $f\in \widetilde W^{l,p}(\Omega \cap V_j )$ and  $\alpha \in \N_0^n$  with $|\alpha|= l$.
\end{lemma} 

\begin{proof} The proof is an easy consequence of Lemma~\ref{morrey}. For the convenience of the reader we write a few details for the proof of (\ref{morrey2elebis}). We assume 
directly that $V_j\cap \partial \Omega \ne \emptyset$ since the other case $V_j\cap \partial \Omega = \emptyset$ is trivial. Let $\alpha, \delta  $ and $f$ be as in the statement and let $B_r$
be a ball in $\R^n$ with radius $r <\delta $.  By changing variables in integrals and applying the chain rule, we immediately deduce from   Corollary~\ref{morrey} that 
\begin{multline*}
\int_{B_r}|D^{\alpha }T_jf |^pdx  =\int_{B_r}|D^{\alpha}   \left(T_{{\mathcal{H}}_j}\circ \Lambda_j(f)\right)  (\lambda_j(x))|^pdx\\
=\int_{\lambda_j(B_r)}|D^{\alpha}   \left(T_{{\mathcal{H}}_j}\circ \Lambda_j(f)\right) (y)|^pdy  \le c \phi (r) \sum_{|\beta   |= l}
\| D^{\beta }(\Lambda_j(f) )\|^p_{M^{\phi ,\delta}_p( {\mathcal{H}}_j)} \\
\le c \phi (r) \sum_{|\beta   |= l}
\| D^{\beta }f \|^p_{M^{\phi , \delta }_p( {\mathcal{H}}_j)} ,
\end{multline*}
which allows to conclude. 
\end{proof}

\begin{remark} \label{remorder} We note that in Lemma~\ref{lemmarot}, Theorem~\ref{morreyele}, and Corollary~\ref{morrey}, inequalities (\ref{morrey2elebis}),  (\ref{morrey2ele}) and (\ref{morrey2}) allow to estimate the derivatives of the extended function $Tf$ of order $|\alpha|=l$ by means of all derivatives of $f$ of order $|\beta |=l$, and this is valid for all functions $f$  in the Sobolev space $W^{l,p}$. Clearly, Burenkov's extension operator defined for functions in $W^{l,p}$ works also for functions in $W^{m,p}$ for any $m\le l$. This implies that all above mentioned inequalities hold also for any $|\alpha |\le l$, provided one replaces in the right-hand sides $|\beta|=l$ by $|\beta|=|\alpha|$.
\end{remark}

Finally, we can prove the following

\begin{theorem}\label{finalthm} Let ${\mathcal{A}}$ be an atlas in $\R^n$, $M>0$ and $\Omega $
be an open set of class $C^{0,1}_M({\mathcal{A}})$. Let $l\in \N$, $p\in [1,\infty [$ and $\phi:]0,+\infty[\to ]0,+\infty[$. 

Then for every $\delta \in ]0,\infty [$ there exists $c>0$ depending only on
$n, {\mathcal{A}}$, $M$, $l,p, \delta$ such that inequality (\ref{finalthm1intro}) holds
for all $f\in W^{l,p}(\Omega)$ and $|\alpha |\le l$.

Moreover, if $\Omega $ is bounded, $c$ can be chosen to be independent of $\delta$,  hence
\begin{equation}\label{finalthm2}
\| D^{\alpha}Tf\|_{M^{\phi  }_p(\R^n)}\le c\sum_{0\le |\beta |\le |\alpha|} \| D^{\beta}f\|_{M^{\phi }_p(\Omega)},
\end{equation}
for all $f\in W^{l,p}(\Omega)$ and $|\alpha |\le l$.
\end{theorem}

\begin{proof}   Let $\delta \in ]0,\infty [$ be fixed and $B_r$ be a ball of radius $r$ with $0<r\le \delta$. We set 
$\tilde S =\{j\in \overline{1,s}:\ B_r\cap V_j\ne  \emptyset \}$, 
$\tilde s=\sharp \tilde S$ and we note that $\tilde s$ is a finite number depending only on ${\mathcal{A}}$
and $\delta$. 

By Lemma~\ref{lemmarot}  and Remark~\ref{remorder} we get
\begin{multline}
\int_{B_r}|D^\alpha Tf|^pdx =\int_{B_r} |D^\alpha (\sum_{j=1}^s\psi_jT_j(f\psi_j))|^pdx\\
\le c\sum_{j\in \tilde S}\int_{B_r\cap V_j}\sum_{\gamma \le \alpha}|D^{\alpha-\gamma}\psi_jD^\gamma T_j(f\psi_j)|^pdx\\
\le c \sum_{j\in \tilde S}\int_{B_r}\sum_{\gamma \le \alpha }|D^\gamma T_j(f\psi_j)|^pdx\\
\le c \phi (r)  \sum_{j\in \tilde S}\sum_{0\le |\beta |\le |\alpha|} \| D^{\beta}f\|^p_{M^{\phi , \delta}_p(\Omega \cap V_j)}\\
\le c \tilde s \phi (r)  \sum_{0\le |\beta |\le |\alpha|} \| D^{\beta}f\|^p_{M^{\phi , \delta}_p(\Omega )},\\
\end{multline}

which implies the validity of (\ref{finalthm1intro}).

If $\Omega$ is bounded then $\tilde s\le s<\infty $, hence $\tilde s$ in the previous inequality can be replaced by $s$ which is independent of $\delta$, and (\ref{finalthm2}) follows.
\end{proof}

{\bf Acknowledgments} The authors are deeply indebted with Professor Victor I. Burenkov for suggesting the problem and useful discussions. 

The second authors acknowledges financial support from the research  projects `Singular perturbation problems for differential operators' Progetto di Ateneo of the University of Padova,
and  `INdAM GNAMPA Project 2015 - Un approccio funzionale analitico per problemi di perturbazione singolare e di
omogeneizzazione'. 

The  authors are also member of the Gruppo Nazionale per l'Analisi Ma\-te\-ma\-ti\-ca, la Probabilit\`{a} e le loro Applicazioni (GNAMPA) of the
Istituto Nazionale di Alta Matematica (INdAM).

Both authors acknowledge the warm hospitality received by each other's institution on the occasion of several research visits.

\end{document}